\documentclass[12pt, reqno]{amsart}
\usepackage{amsmath, amsthm, amscd, amsfonts, amssymb, graphicx, color}
\usepackage[bookmarksnumbered, colorlinks, plainpages]{hyperref}

\newtheorem{theorem}{Theorem}[section]
\newtheorem{lemma}[theorem]{Lemma}
\newtheorem{proposition}[theorem]{Proposition}
\newtheorem{corollary}[theorem]{Corollary}
\theoremstyle{definition}

\theoremstyle{remark}
\newtheorem{remark}[theorem]{Remark}
\numberwithin{equation}{section}

\begin{document}
\setcounter{page}{1}

\title[Amenability and harmonic $L^p$-functions  on hypergroups ]{Amenability and harmonic $L^p$-functions  on hypergroups }

\author[M. Nemati and J. Sohaei]{ Mehdi Nemati$^1$ and Jila Sohaei$^2$}
\address{$^{1}$Department of Mathematical Sciences,
     Isfahan Uinversity of Technology, Isfahan 84156-83111, Iran;
     \newline
 School of Mathematics,
      Institute for Research in Fundamental Sciences (IPM),
      P.O. Box: 19395--5746, Tehran, Iran.}
\email{\textcolor[rgb]{0.00,0.00,0.84}
{m.nemati@cc.iut.ac.ir}}
\address{$^{2}$ Department of Mathematical Sciences,
	Isfahan Uinversity of Technology,
	Isfahan 84156-83111, Iran}
\email{\textcolor[rgb]{0.00,0.00,0.84}{j.sohaei@math.iut.ac.ir}}





\begin{abstract}
Let $K$
be  a locally compact hypergroup with a left invariant  Haar measure. We show that the Liouville property and amenability are equivalent for $K$ when it is second countable. Suppose that $\sigma$ is a non-degenerate probability measure on $K$, we show that  there is no non-trivial
$\sigma$-harmonic function which is continuous and vanishing at infinity. Using this, we prove that  the space $H_\sigma^p(K)$ of all $\sigma$-harmonic $L^p$-functions, is trivial for all $1\leq p<\infty$.  Further,  it is shown that $H_\sigma^\infty(K)$ contains only constant functions  if and only if it is a subalgebra of $L^\infty(K)$.   In the case where $\sigma$ is adapted and $K$ is compact, we show that $H_\sigma^p(K)={\mathbb C}1$ for all $1\leq p\leq\infty$.
\end{abstract} \maketitle

\section{Introduction}
Let $\sigma$ be  a complex  Borel measure on  a locally compact group $G$. A
Borel function $f$ on $G$ 
is called {${\sigma}$-harmonic} if   it satisfies the convolution equation $\sigma*f=f$. 
It is a well-known result
of \cite{cd1} that if $G$ is abelian, then the only
bounded continuous $\sigma$-harmonic function are
constant functions when the support of $\sigma$ generates a dense subgroup of $G$. Bounded  harmonic functions have been investigated
by several authors for various kinds of groups, e.g., nilpotent
groups  and compact groups \cite{ch, cl, dm, jw, joh}.  Moreover, it was shown in \cite{chu3} that for $1\leq p<\infty$, any $\sigma$-harmonic $L^p$-function associated to an
adapted probability measure $\sigma$ on a locally compact group $G$  is trivial. 
Harmonic functions on groups play  important roles in analysis, geometry and probability
theory \cite{chu}.  

Motivated
by these observations,  bounded continuous harmonic functions on nilpotent, [IN] and  central hypergroups
have been studied in \cite{ac2, ac}.

In what follows, $K$  denotes a locally  compact hypergroup with a left-invariant Haar measure. The purpose
of this paper is to obtain some insight into the harmonic functions problem for
the $L_p$-spaces, $1\leq p\leq \infty$, of $K$.


In Section 3, for given a complex  Borel measure  $\sigma$ on  $K$  with $\|\sigma\|=1$, we first show that there is a contractive projection from $L^p(K)$, $1< p\leq \infty$,  onto $H^p_\sigma(K)=\{f\in L^p(K): \sigma*f=f\}$.
We also show that  $K$ is necessarily amenable if it has the Liouville property; that is, there exists a  probability measure $\sigma$ on  $K$  such that all $\sigma$-harmonic $L^\infty$-functions on $K$
are constant. Further, we prove that a second countable hypergroup possesses the Liouville property if and only if it is amenable.

In Section 4, for the case that $\sigma$ is a non-degenerate probability measure on $K$, we show that the space of all $\sigma$-harmonic functions which are continuous and vanishing at infinity are trivial. Using this we prove that  for $1\leq p<\infty$, any $\sigma$-harmonic $L^p$-function  is trivial.  For such a measure $\sigma$, we also prove that $H_\sigma^\infty(K)$ is a subalgebra of $L^\infty(K)$ if and only if $H_\sigma^\infty(K)=\mathbb{C}1$. In the case where $\sigma$ is adapted and $K$ is compact, we show that $H_\sigma^p(K)={\mathbb C}1$ for all $1\leq p\leq\infty$. These extend
the results for the group case in \cite{chu3}.

\section{Preliminaries}
\noindent  Let $K$ be a locally compact Hausdorff space. The space $K$ is a hypergroup
if there exists  a bilinear, associative, weakly
continuous convolution $*$ on the Banach space ${ M}(K)$ of all bounded regular complex valued
Borel measures on $K$,  such that $({ M}(K), *)$ is an algebra and satisfies,
for $x,y\in K$,

(i) $\delta_x*\delta_y$ is a probability measure on $K$ with compact support,

(ii) the mapping $K\times K\rightarrow {\mathcal C}(K)$ , $(x, y)\mapsto {\rm supp}(\delta_x*\delta_y)$ is continuous with respect to the Michael
topology on the space ${\mathcal C}(K)$ of nonvoid compact sets in $K$,

(ii) the mapping $K\times K\rightarrow { M}(K)$, $(x, y)\mapsto \delta_x*\delta_y$ is continuous,

(ii) there is an identity $e\in K$ with $\delta_e*\delta_x=\delta_x*\delta_e=\delta_x$, 

(iv)there is a continuous involution $\check{}$ on $K$ such that
$(\delta_x*\delta_y)\check{}=\delta_{\check{y}}*\delta_{\check{x}}$
and $e\in {\rm supp}(\delta_x*\delta_y)$ if and only if $x=\check{y}$. The image measure of $\mu\in{ M}(K)$ under
such involution is denoted by
$\check{\mu}$.

Given a (complex) Borel function $f$ on $K$
and $x, y\in K$ the left translation $_xf$ and the right translation $f_y$  are
defined by
$$
_xf(y)=f_y(x)=\int_K f(t) d(\delta_x*\delta_y)(t)=f(x*y),
$$
if the integral exists, where $f(x*y)=\int_Kf d(\delta_x*\delta_y)$. For a  Borel function $f$ on $K$ the Borel function $\check{f}$ is defined by $\check{f}(x)=f(\check{x})$ for all $x\in K$. Given $\mu, \nu\in { M}(K)$, their convolution is given by
$$
\langle \mu*\nu, f\rangle=\int_Kf~d(\mu*\nu)=\int_K\int_Kf(x*y)~d\mu(x)~d\nu(y)\quad(f\in C_0(K))
$$
and $\|\mu*\nu\|\leq \|\mu\|\nu\|$ which shows that $\mu*\nu\in { M}(K).$
Also for a measure $\mu\in M(K)$ and a  Borel function $f$ on $K$, we define the convolutions
$\mu*f$ and $f*\mu$ by 
\begin{eqnarray*}
	\mu*f(x)=\int_K f(\check{y}*x)~d\mu(y), \quad f*\mu(x)=\int_K f(x*\check{y})~d\mu(y)\quad(x\in K),
\end{eqnarray*}
if the integrals exist. 
Note that in this case $(\mu*f)\check{}=\check{f}*\check{\mu}$. Moreover, if $f$ is  in $C_b(K)$, the Banach space of bounded complex
continuous functions on $K$, then $\mu*f$ and $f*\mu$ are in $C_b(K)$ with $\|\mu*f\|_\infty\leq\|\mu\|\|f\|_\infty$  and $\langle\mu*\nu, f\rangle=\langle\nu, \check{\mu}*f\rangle$. We refer the reader to \cite{bloom} for details of hypergroups.

\section{Amenability and Liouville property}
Throughout  of this paper, let $K$ be a locally
compact hypergroup with a left-invariant Haar measure $\omega$; that is, a non-zero positive Radon measure on $K$ such that
$$
\delta_x*\omega=\omega\quad(x\in K).
$$
Let $C_0(K)$ be the Banach space of complex
continuous functions on $K$ vanishing at infinity. Then its dual identifies, via the
Riesz representation theorem, with the space $M(K)$.
Let $L^p(K)$ be the complex
Lebesgue spaces with respect to $\omega$, for $1\leq p\leq\infty$. Given a Borel measure $\sigma$ on a hypergroup $K$, a
Borel function $f$ on $K$ satisfying the convolution equation
$$
\sigma*f=f
$$
is called {\it ${\sigma}$-harmonic}. For $1\leq p\leq\infty$ define $H_\sigma^p(K)$ to be the set of all ${\sigma}$-harmonic $L^p$-functions; that is, $H_\sigma^p(K)=\{f\in L^p(K): \sigma*f=f\}$.
For Borel functions $f$ and  $g$ at
least one of which is $\sigma$-finite, define the convolution $f*g$   on $K$ by
$$(f*g)(x)=\int_Kf(y)g(\check{y}*x)d \omega(y).$$
We commence with the following lemma whose proof is similar  to those given in \cite{chu3}. For completeness, we present the argument here. 
\begin{lemma}\label{th0}
	Let $\sigma\in M(K)$ with $\|\sigma\|=1$ and let $1< p\leq \infty$. Then there is a contractive projection $P_\sigma: L^p(K) \rightarrow L^p(K)$ with $P_\sigma(L^p(K))=H_\sigma^p(K)$. Moreover, if $1<p, q<\infty$ with $\frac{1}{p}+\frac{1}{q}=1$, then $P_\sigma$ is the dual map of the projection $P_{\check{\sigma}}: L^q(K) \rightarrow L^q(K)$.
	
\end{lemma}
\begin{proof}
	Let ${\mathcal U}$ be a free ultra-filter  on $\mathbb{N}$, and define 
	$
	P_\sigma: L^p(K) \rightarrow L^p(K)
	$
	by the weak$^*$ limit
	$$
	P_\sigma(f)=\lim_{\mathcal U}\frac{1}{n}\sum_{k=1}^n \sigma^k*f,
	$$
	where $\sigma^k$ is the $k$-times convolution of $\sigma$ with itself. It is easy to see that $P_\sigma(f)=f$ for all $f\in H_\sigma^p(K)$. Moreover, if $f\in L^p(K)$, then it is easily verified that $\sigma*P_\sigma(f)=P_\sigma(f)$ and so $P_\sigma(f)\in H_\sigma^p(K)$. These show that $P_\sigma^2=P_\sigma$ and $P_\sigma(L^p(K))=H_\sigma^p(K)$.
	
	Suppose now that $1<p<\infty$. Then it is not hard to check that $\sigma*P_\sigma(f)=P_\sigma(\sigma*f)$ for all $f\in L^p(K)$. 
	Therefore, for each $g\in L^q(K)$, we have
	\begin{eqnarray*}
		\langle \sigma* P_{\check{\sigma}}^*(f), g\rangle&=&\langle  P_{\check{\sigma}}^*(f), \check{\sigma}*g\rangle=
		\langle f, P_{\check{\sigma}}(\check{\sigma}*g)\rangle\\
		&=&\langle f, \check{\sigma}* P_{\check{\sigma}}(g)\rangle
		=\langle f, P_{\check{\sigma}}(g)\rangle\\
		&=&\langle P_{\check{\sigma}}^*(f), g\rangle.
	\end{eqnarray*}
	This shows that $\sigma* P_{\check{\sigma}}^*(f)=P_{\check{\sigma}}^*(f)$ for all $f\in L^p(K)$. Similarly, we can show that $\check{\sigma}* P_{{\sigma}}^*(g)=P_{{\sigma}}^*(g)$ for all $g\in L^q(K)$. Consequently, for each $f\in L^p(K)$ and $g\in L^q(K)$, we have
	\begin{eqnarray*}
		\langle P_\sigma P_{\check{\sigma}}^*(f), g\rangle=\langle  P_{\check{\sigma}}P_\sigma^*(g), f\rangle&=&
		\langle  P_{{\sigma}}^*(g), f\rangle\\
		&=&\langle  P_{{\sigma}}(f), g\rangle.
	\end{eqnarray*}
	This shows that $P^*_{\check{\sigma}}(f)=P_\sigma P_{\check{\sigma}}^*(f)=P_\sigma(f)$ for all $f\in L^p(K)$, as required.
\end{proof}

\begin{remark}\label{th2}
	Let $\sigma\in M(K)$ with $\|\sigma\|=1$ and let $1< p, q< \infty$ be such that $\frac{1}{p}+\frac{1}{q}=1$. Then we have linear
	isometric isomorphisms $H_\sigma^p(K)\cong L^p(K)/H_{\check{\sigma}}^q(K)^\perp\cong H_{\check{\sigma}}^q(K)^*$, where the first isometry is given by $H_\sigma^p(K)\ni f\mapsto f+H_{\check{\sigma}}^q(K)^\perp\in L^p(K)/H_{\check{\sigma}}^q(K)^\perp$. Indeed, for each $f\in H_\sigma^p(K)$, we have
	\begin{eqnarray*}
		\|f\|\geq \inf \{\|f+g\|: g\in H_{\check{\sigma}}^q(K)^\perp\}&\geq& \inf\{\|P_\sigma(f+g)\|: g\in H_{\check{\sigma}}^q(K)^\perp\}\\
		&\geq&
		\inf\{\|f\|: g\in H_{\check{\sigma}}^q(K)^\perp\}= \|f\|.
	\end{eqnarray*}
\end{remark}

Recall that the hypergroup $K$ is called  amenable if there exists a topological left invariant mean
on $L^\infty(K)$; that is, there exists $m\in L^\infty(K)^*$ such that $\|m\|=m(1)=1$ and $m(g*f)=(\int_K g~d\omega)m(f)$ for all $f\in L^\infty(K)$ and $g\in L^1(K)$. A topological right invariant
mean on $L^\infty(K)$ is a functional $m\in L^\infty(K)^*$ such that $\|m\|=m(1)=1$ and $m(f*\check{g})=(\int_K g~d\omega)m(f)$ for all $f\in L^\infty(K)$ and $g\in L^1(K)$.
It is known that the involution on $L^1(K)$ can be canonically extended to a
linear involution $\star$ on $L^1(K)^{**}$; see \cite[Chapter 2]{dl}. Clearly, $m\in L^\infty(K)^*$
is a topological left invariant mean if and only if $m^\star$ is a topological right invariant mean. Therefore, the existence of a topological right invariant mean on $L^\infty(K)$ is equivalent to $K$ being amenable.

\begin{theorem}\label{ame}
	Let $K$ be a hypergroup with the Liouville property; that is, there exists a probability measure $\sigma$ on $K$ such that $H_\sigma^\infty(K)=\mathbb{C}1$. Then $K$ is amenable.
\end{theorem}
\begin{proof}
	Let $P_\sigma: L^\infty(K) \rightarrow H_\sigma^\infty(K)$ be the contractive projection as defined in Lemma \ref{th0}. Then there is a unique functional $m\in L^\infty(K)^*$ such that $P_\sigma(f)=m(f)1$ for all $f\in L^\infty(K)$.
	Since $\sigma*(f*\check{g})=(\sigma*f)*\check{g}$ for all $f\in L^\infty(K)$ and $g\in L^1(K)$, it follows that $P_\sigma(f*\check{g})=P_\sigma(f)*\check{g}$.  Moreover, since  the projection $P_\sigma$ is positive and $P_\sigma(1)=1$, we conclude that $\|m\|=m(1)=1$. This shows that $m$ is a topological right invariant mean on $L^\infty(K)$, which implies that $K$ is amenable.
	
\end{proof}
For a a   locally compact hypergroup $K$ consider the  closed two sided ideal  $$L^1_0(K)=\left\{f\in L^1(K):~\int_Kf d\omega=0\right\}$$  in $L^1(K)$ and for each  $\sigma\in M(K)$  let $J_\sigma$ be the norm closure of $\{ f-\check{\sigma}*f:~f\in L^1(K)\}$ in $L^1(K)$. 
It is well known that
$L^1_0(K)$ has codimension one in $L^1(K)$ and if $\sigma$ is a probability measure, then  $J_\sigma\subseteq L^1_0(K)$.
Moreover, it is easy to see that $J_\sigma^\perp=\{f\in L^\infty(K): \sigma*f=f\}=H_\sigma^\infty(K)$ and hence $H^\infty_\sigma(K)=({L^1(K)}/{J_\sigma})^*$

We have the following lemma  whose proof is similar  to those given in  \cite[Lemma 1.1 and Remark 3,  p.210]{wil} for locally compact groups. 
Thus, we
omit  the proof.
\begin{lemma}\label{xb}
	Let $K$ be a   locally compact hypergroup  and ${\mathcal S}$ be a norm closed, convex
	subsemigroup of probability measuers on $K$. Let $I$ be a separable, closed
	subspace of $L^1(K)$ such that
	
	\emph{(i)} ${ J}_\sigma\subseteq I$ for every $\sigma\in{\mathcal S}$; and
	
	\emph{(ii)} for each $\varepsilon>0$ and $g\in I$ there is
	$\sigma\in{\mathcal S}$ such that $$d(g, { J}_\sigma)=\inf\{\|f-g\|:~f\in{J}_\sigma \}<\varepsilon.$$
	Then there is $\sigma\in {\mathcal S}$ such that $I={J}_\sigma$.
\end{lemma}

\begin{corollary}\label{cor}
	Let $K$ be a second countable locally compact hypergroup. Then the following conditions are equivalent.
	
	\emph{(i)} $K$ is amenable.
	
	\emph{(ii)} There exists a probability measure $\sigma$ on $K$ such that $L^1_0(K)=J_\sigma$.
	
	\emph{(iii)} $K$ has the Liouville property.
\end{corollary}
\begin{proof}
	{(i)}$\Rightarrow${(ii)}.
	Suppose that  $K$  is amenable. Then by \cite[Corollary 4.2]{skan}, there is a net $(f_\alpha)$ in $P_1(K):=\{f\in L^1(K): \|f\|_1=\int_Kf~d\omega=1\}$ such that $$\|f_\alpha*f-f_\alpha\|_1\rightarrow 0$$ for all $f\in P_1(K)$. In particular, for each $g\in L^1_0(K)$ we have $\|f_\alpha*g\|_1\rightarrow 0$.  Moreover, $f_\alpha*g-g\in J_{\check{\sigma}_\alpha}$ for all $\alpha$, where $\sigma_\alpha=f_\alpha\omega$. This shows that the condition (ii) of Lemma \ref{xb} is
	satisfied. Since  $L^1(K)$ is separable, we give that $L^1_0(K)=J_\sigma$ for some probability measure $\sigma$ on $K$.
	
	{(iii)}$\Rightarrow${(i)}.  This follows from Theorem \ref{ame}.
	
	{(ii)}$\Leftrightarrow${(iii)}. This follows from the inclusion $J_\sigma\subseteq L^1_0(K)$ with the fact that $L^1_0(K)^\perp={\Bbb C}1$.
\end{proof}

\begin{proposition}\label{pos}
	Let $\sigma$ be a probability measure on $K$ and let $1<p\leq\infty$. Then  $H_\sigma^p(K)$ is generated by its non-negative elements.
\end{proposition}
\begin{proof}
	Suppose that  $f\in H_\sigma^p(K)$. Then $\sigma*\overline{f}=\overline{(\sigma*f)}=\overline{f}$. This shows that $H_\sigma^p(K)$ is self-adjoint and consequently 
	is generated by its real function parts. Now let $f\in H_\sigma^p(K)$ be a real function and let $f=f_+-f_-$, where $f_+, f_-$ are non-negative functions  in $L^p(K)$. Since $\sigma$ is positive, the projection $P_\sigma$, as defined in Theorem \ref{th0}, is positive. It follows that $P_\sigma(f_+), P_\sigma(f_-)\in H_\sigma^p(K)$ are non-negative. Moreover, $f=P_\sigma(f)=P_\sigma(f_+)-P_\sigma(f_-)$, and this completes the proof.
\end{proof}

\section{Harmonic $L^p$-functions}
Let $\mu$ be  a complex  Borel measure on  a locally compact hypergroup $K$. We say that $\mu$ is non-degenerate if   $$K=\overline{\bigcup_{n=1}^\infty({\rm supp}|\mu|)^n}=\overline{\bigcup_{n=1}^\infty{\rm supp}|\mu|^n},$$
where $|\mu|^n$ is the $n$-fold convolution of $|\mu|$ and ${\rm supp}|\mu|^n$ equals the closure of $({\rm supp}|\mu|)^n$. If $\mu$ satisfies the weaker condition that 
$$K=\overline{\bigcup_{n=1}^\infty{\Big(}{\rm supp}|\mu|\cup ({{\rm supp}|\mu|})\check{}{\Big)}^n}$$
then  we say that $\mu$ is adapted. 
\begin{remark}\label{rem}
	{\rm 
		Let $\mu\in M(K)$. Then it is not hard to check that  non-degeneracy of $\mu$ is equivalent to that  $\sum_{n=1}^\infty\langle |\mu|^n, h\rangle>0$ for every non-zero $h\in C_c(K)^+$, or equivalently there exists  $n\in \mathbb{N}$  such that
		$$
		\langle |\mu|^n, h\rangle>0.
		$$
		Therefore, if  $f\in C_b(K)^+$ is non-zero, then we may find $h\in C_c(K)^+$ such that $\|h\|_\infty=1$ and 
		$fh\neq 0$. It follows that $fh\in C_c(K)^+$ and $fh\leq f$. Therefore,  there exists  $n\in \mathbb{N}$  such that
		$$
		0<\langle |\mu|^n, fh\rangle\leq \langle |\mu|^n, f\rangle.
		$$
	}
\end{remark}

\begin{theorem}\label{sub}
	Let $\sigma$ be a probability measure on $K$. Then the following statements are equivalent.
	
	\emph{(i)} $H^\infty_\sigma(K)$ is a subalgebra of $L^\infty(K)$.
	
	\emph{(ii)} $H^\infty_\sigma(K)$ is a von Neumann subalgebra of $L^\infty(K)$..
	
	\emph{(iii)} $H^\infty_\sigma(K)=\{f\in L^\infty(K): \forall x\in K, f(\check{y}*x)=f(x),~ \hbox{for  $\sigma-$a.e.}~ y\in K \}$.
\end{theorem}

\begin{proof}
	Since $H^\infty_\sigma(K)$ is a weak$^*$ closed operator system, (i) implies that $H^\infty_\sigma(K)$ is
	a von Neumann subalgebra of $L^\infty(K)$. The implication (iii)$\Rightarrow$(i) is trivial. We need to prove (ii)$\Rightarrow$(iii).   Let $x \in K$ and $f\in H^\infty_\sigma(K)$. Without loss of generality assume that $f$ is real valued. Since $(f-f(x))^2\in H^\infty_\sigma(K)$, it follows that
	\begin{eqnarray*}
	\int_{K} (f(\check{y}\ast x)-f(x))^2d\sigma(y)&=&\sigma*(f-f(x))^2(x)\\
	&=&(f(x)-f(x))^2=0.
	\end{eqnarray*}
	This implies that $f(\check{y}\ast x)=f(x)$ for $\sigma-$almost every $y\in K$.
\end{proof}
For a hypergroup $K$, we denote by $LUC(K)$ to be  the Banach space of all bounded left uniformly
continuous complex functions on $K$, consisting of bounded continuous function $f$ on
$K$ such that the map $K\ni x\mapsto f_x\in C_b(K)$ is continuous.
\begin{lemma}
	Let $\sigma\in M(K)$. Then $H_\sigma^\infty(K)\cap LUC(K)$ is weak$^*$ dense in $H_\sigma^\infty(K)$.
\end{lemma}
Let $(\phi_\alpha)$ be a bounded approximate identity for $L^1(K)$ and $f\in H^\infty_\sigma(K)$. Then $f*\check{\phi_\alpha}\in H^\infty_\sigma(K)\cap LUC(K)$ for all $\alpha$, by \cite[Lemma 2.2]{skan}. Moreover, for each $g\in L^1(K)$, we have
$$
\lim_\alpha\langle f*\check{\phi_\alpha}, g\rangle=\lim_\alpha\langle f, g*\phi_\alpha\rangle=\langle f, g\rangle.
$$
Thus, $H_\sigma^\infty(K)\cap LUC(K)$ is weak$^*$ dense in $H_\sigma^\infty(K)$.
\begin{corollary}
	Let $\sigma$ be a non-degenerate probability measure on $K$. Then the following conditions are equivalent.
	
	{\rm (i)}   $H_\sigma^\infty(K)$ is a subalgebra of $L^{\infty}(K)$.
	
	{\rm (ii)} $H_\sigma^\infty(K)= \mathbb{C}{1}$.
\end{corollary}
\begin{proof}
	Suppose that (i) holds. Given $f\in H_\sigma^\infty(K)\cap LUC(K)$, by Theorem \ref{sub} for each  $n\in\mathbb{N}$, we have
	$\check{f}(y)=\check{f}(e)$ for all $y\in (\rm{supp} \sigma)^n$. It follows
	from non-degeneracy of $\sigma$ and continuity of $f$ that
	$f$ is constant. Since $H_\sigma^\infty(K)\cap LUC(K)$ is weak$^*$ dense in $H_\sigma^\infty(K)$, we give that $H_\sigma^\infty(K)= \mathbb{C}{1}$.
\end{proof}
A subspace $X$ of $L^p(K)$, $1\leq p\leq\infty$, is called left (resp. right) translation invariant if $_xf\in X$ (resp. $f_x\in X$) for all $f\in X$ and $x\in K$. The subspace $X$ is called translation invariant if it is left and right translation invariant. It is easy to check that for each $\sigma\in M(K)$ the space  $H^\infty_\sigma(K)$ is a right translation invariant subspace of $L^\infty(K)$. Since $\int_K (g_x)(t)f(t) d\omega(t)=\frac{1}{\Delta(x)}\int_Kg(t)(f_{\check{x}})(t)  d\omega(t)$  for all $f\in L^\infty(K)$, $g\in L^1(K)$ and $x\in K$, it follows that $J_\sigma$ is also  right translation invariant in $L^1(K)$, where $\Delta$ is the modular function on $K$. We recall that $L^\infty(K)$ is naturally  a Banach $L^1(K)$-bimodule by the following module actions
$$
\langle g\cdot f, h\rangle =\langle f, h*g\rangle,\quad \langle f\cdot g, h\rangle=\langle  f, g*h\rangle\quad (f\in L^\infty(K), g, h\in L^1(K)).
$$
It is easily
verified that  
$$(g\cdot f)(x)=\int_K g(y)( _xf)(y) d\omega(y),\quad (f\cdot g)(x)=\int_K g(y)(f_x)(y) d\omega(y)$$
for all $x\in K$.
\begin{proposition}\label{alg}
	Let $\sigma \in M(K)$. Then the following conditions are equivalent.
	
	{\rm(i)} $H^\infty_\sigma(K)$ is translation invariant.
	
	{\rm(ii)} $J_\sigma$ is translation invariant.
	
	{\rm(iii)} $J_\sigma$ is an ideal in $L^1(K)$.
	
	{\rm(iv)} $\int_K f(\check{y}*x) d\sigma(y)=\int_K f(x*\check{y}) d\sigma(y)$  for all $f\in H^\infty_\sigma(K)\cap LUC(K)$.
	
	{\rm(v)} $H^\infty_\sigma(K)$ is a sub-$L^1(K)$-bimodule of $L^\infty(K)$.
\end{proposition}
\begin{proof}
	(i)$\Leftrightarrow$(ii).
	Since $\int_K  ({_x}g)(t)f(t) d\omega(t)
	=\frac{1}{\Delta(x)}\int_Kg(t)( _{\check{x}}f)(t)  d\omega(t)$ 
	for all $f\in L^\infty(K)$, $g\in L^1(K)$ and $x\in K$, it follows that $H^\infty_\sigma(K)$ is left translation invariant if and only if $J_\sigma$ is.
	
	(ii)$\Rightarrow$(iii). It suffices to show that  $J_\sigma$ is a left ideal in $L^1(K)$. To prove this, given $g\in L^1(K)$, $h\in J_\sigma$ and $f\in H_\sigma^\infty(K)$, we have
	\begin{eqnarray*}
	\langle f, g*h\rangle&=&\int_K f(x)(\int_Kg(y)h(\check{y}*x) d\omega(y))d\omega(x)\\
	&=&\int_Kg(y)(\int_Kf(x)( _{\check{y}}h)(x) d\omega(x))d\omega(y)=0,
\end{eqnarray*}
	which implies that $g*h\in J_\sigma$. 
	
	(iii)$\Rightarrow$(ii). Let $(\phi_\alpha)$ be a bounded approximate identity for $L^1(K)$ and let $g\in J_\sigma$. Since $(( _{x}\phi_\alpha)*g)=  _{x}(\phi_\alpha*g)$ and $_{x}\phi_\alpha\in L^1(K)$ for all $\alpha$ and $x\in K$, the proof follows from the fact that $\phi_\alpha*g\rightarrow g$.
	
	(i)$\Rightarrow$(iv). Suppose that $f\in H^\infty_\sigma(K)\cap LUC(K)$. Since  $H^\infty_\sigma(K)$ is left translation invariant, we obtain that 
	\begin{eqnarray*}
		\int_K f(\check{y}*x) d\sigma(y)&=&{_{x}}(\sigma*f)(e)\\
		&=&( _{x}f)(e)=(\sigma*( _{x}f))(e)\\
		&=&\int_K f(x*\check{y}) d\sigma(y).
	\end{eqnarray*}
	
	(iv)$\Rightarrow$(i). Suppose that $f\in H^\infty_\sigma(K)\cap LUC(K)$ and $x\in K$. Then $f_x\in  H^\infty_\sigma(K)\cap LUC(K)$, by right translation invariance of $H^\infty_\sigma(K)\cap LUC(K)$. Moreover, for each $y\in K$, we have
	\begin{eqnarray*}
		( _{x}f)(y)= _{x}(\sigma*f)(y)&=&\int_K f(\check{t}*x*y) d\sigma(t)\\
		&=& \int_K (f_y)(\check{t}*x) d\sigma(t)\\
		&=& \int_K (f_y)(x*\check{t}) d\sigma(t)\\
		&=&\int_K f(x*\check{t}*y) d\sigma(t)\\
		&=&(\sigma*( _{x}f))(y).
	\end{eqnarray*}
	This shows that $H^\infty_\sigma(K)\cap LUC(K)$ and hence  $H^\infty_\sigma(K)$ is left translation invariant by weak$^*$ density of $H^\infty_\sigma(K)\cap LUC(K)$ in $H^\infty_\sigma(K)$.
	
	(i)$\Rightarrow$(v). Let $f\in H_\sigma^\infty(K)$ and $g\in L^1(K)$.  As $g\cdot f=f*\check{g}$ , we give that $g\cdot f\in H^\infty_\sigma(K)$. Moreover, by assumption $_{t}f\in H_\sigma^\infty(K)$ for all $t\in K$. Thus, 
	\begin{eqnarray*}
		(\sigma*(f\cdot g))(x)&=&\int_K(f\cdot g)(\check{y}*x) d\sigma(y)\\
		&=&\int_K\int_K g(t)(f_{\check{y}*x})(t) d\sigma(y)d\omega(t)\\
		&=&\int_K\int_K g(t)( _{t}f)({\check{y}*x}) d\sigma(y)d\omega(t)\\
		&=&\int_K g(t)( _{t}f)(x) d\omega(t)=(g\cdot f)(x). 
	\end{eqnarray*}
	This implies that $f\cdot g\in H^\infty_\sigma(K)$.
	
	(v)$\Rightarrow$(iii). It suffices to show that $J_\sigma$ is a left ideal in $L^1(K)$. Indeed, given $f\in H_\sigma^\infty(K)$,  $g\in L^1(K)$ and $h\in J_\sigma$, we have
	$$
	\langle f, g*h\rangle=\langle f\cdot g, h\rangle=0,
	$$
	which yields that $g*h\in J_\sigma$, as required.
\end{proof}

\begin{remark}
	It is obvious that under each of above equivalent conditions in Proposition	\ref{alg}, the quotient space $L^1(K)/J_\sigma$ is a Banach algebra. This implies that $H^\infty_\sigma(K)^*=({L^1(K)}/{J_\sigma})^{**}$ is a Banach algebra with respect to the two Arens products.
\end{remark}	
\begin{theorem}\label{3.6}
	Let $\sigma$ be a non-degenerate probability measure on $K$. Then every bounded continuous $\sigma$-harmonic function on $K$ vanishing at infinity is constant.
\end{theorem}
\begin{proof}
	Let $f\in {H}_\sigma^\infty(K)\cap C_0(K)$ be real-valued. Without loss of generality assume that $\|f\|_\infty=1$. Therefor, we can find a probability measure $\mu$ on $K$ such that $\|f\|_\infty=\langle \mu, f\rangle$. If $f\neq 1$, then  the function $1-f$ is also non-negative and non-zero $\sigma$-harmonic function in $C_b(K)$. It is well known from \cite[Proposition 1.2.16]{bloom} that  $\mu*(1-\check f)$ is a non-negative and bounded  continuous function. Moreover, $$\int_K\mu*(1-\check f)d\omega(x)=\mu(K)\int_K(1-\check f)d\omega(x)>0,$$ which implies that $\mu*(1-\check f)$ is non-zero. Since $\sigma$ is non-degenerate, by Remark \ref{rem}, there exists $n\in \mathbb{N}$  such that
	$$
	\langle \sigma^n, \mu*(1-\check f)\rangle>0.
	$$
	On the other hand, since $f$ is $\sigma$-harmonic,  $\sigma^n*f=f$. Therefore, 
	\begin{eqnarray*}
		\langle \sigma^n, \mu*(1-\check f)\rangle&=&\langle \sigma^n, (1-\mu*\check f)\rangle=1-\langle \sigma^n, \mu*\check f\rangle\\
		&=&1-\langle \check{\sigma^n}, f*\check \mu\rangle
		=1-\langle \check{\sigma^n}*\mu,  f\rangle\\
		&=&1-\langle \sigma^n*f, \mu\rangle=1-\langle f, \mu\rangle=0
	\end{eqnarray*}
	which is a contradiction.  Since ${H}_\sigma^\infty(K)\cap C_0(K)$ is generated by its non-negative elements, the proof is complete.
\end{proof}

\begin{corollary}
	Let $K$ be non-compact and let $\sigma$ be a non-degenerate probability measure on $K$. Then the sequence 
	$\left(\frac{1}{n}\sum_{k=1}^n \sigma^k\right)$ is weak$^*$-convergent  to $0$. 
\end{corollary}
\begin{proof}
	Let $\sigma_0$ be a weak$^*$ cluster point of $(\frac{1}{n}\sum_{k=1}^n \sigma^k)$ in $M(K)$. Then $\sigma*\sigma_0=\sigma_0$, which implies that $\sigma_0$ is an idempotent probability measure in $M(K)$. Moreover, for each $f\in C_0(G)$ we have 
	$$
	\sigma*(\sigma_0*f)=(\sigma*\sigma_0)*f=\sigma_0*f.
	$$
	This shows that $\sigma_0*f$ is $\sigma$-harmonic and so it is  constant by Theorem \ref{3.6}. Since $K$ is non-compact, we must have $\sigma_0*f=0$  for all $f\in C_0(K)$. This implies that $\sigma_0=0$. Thus, $0$ is the only weak$^*$ cluster point of $\left(\frac{1}{n}\sum_{k=1}^n \sigma^k\right)$ in $M(K)$. By weak$^*$ compactness of the unit ball of $M(K)$ we conclude that 
	$\frac{1}{n}\sum_{k=1}^n \sigma^k\stackrel{w^*}\longrightarrow 0$.
\end{proof}

\begin{corollary}
	The hypergroup  $K$ is compact if and only if there is a  non-degenerate idempotent probability measure on $K$.  
\end{corollary}

\begin{theorem}
	Let $K$ be non-compact and let $\sigma$ be a non-degenerate probability measure on $K$. Then we have $H_\sigma^p(K)=\{0\}$ for all $1<p<\infty$.
\end{theorem}
\begin{proof}
	Let $1<p<\infty$ and let $f$ be a non-negative function in $H_\sigma^p(K)$ with $\|f\|_p=1$. Consider the probability measure $\sigma_0$ on $K$ defined by
	$$
	\sigma_0:=\sum_{n=1}^\infty\frac{1}{2^n}\sigma^n.
	$$
	It is easy to see that $\sigma_0*f=f$. Moreover, we may find $g\in L^q(K)$ with $\|g\|_q\leq1$ such that $\langle f, g\rangle=\|f\|_p=1$, where $1<q< \infty$ and $\frac{1}{p}+\frac{1}{q}=1$. It follows from \cite[(1.4.11), (1.4.12)]{bloom} that 
	$g*\check{f}\in C_0(K)$ and $\|g*\check{f}\|_\infty\leq \|g\|_q\|f\|_p\leq1$. Therefore, $1-g*\check{f}\geq 0$ and 
	\begin{eqnarray*}
		\langle\sigma_0, 1-g*\check{f}\rangle&=&1-\langle\sigma_0, g*\check{f}\rangle\\
		&=&1-\langle \sigma_0*f, g\rangle\\
		&=&1-\langle f, g\rangle=0.
	\end{eqnarray*}
	It follows from non-degeneracy of $\sigma$ that $1=g*\check{f}\in C_0(K)$, contradicting $K$ being non-compact. Hence, $f=0$. Thus, Proposition \ref{pos} implies that $H_\sigma^p(K)=\{0\}$.
\end{proof}
We use the  following result  which is proved in \cite[Theorem 3.8]{ac} to show that if $\sigma$ is an adapted probability
measure on compact hypergroup $K$, then  each $\sigma$-harmonic  $L^p$-function is trivial for all $1\leq p\leq \infty$. 
\begin{theorem}\label{ami}
	If $\sigma$ is an adapted probability measure on a compact hypergroup $K$, then
	each $\sigma$-harmonic continuous function on $K$ is constant.
\end{theorem}
\begin{theorem}\label{comp}
	Let $K$ be a compact hypergroup and let $\sigma$ be an adapted probability
	measure on $K$. Then for $1\leq p\leq\infty$, we have $H_\sigma^p(K)=\mathbb{ C}1$.
\end{theorem}
\begin{proof}
	Let $K$ be compact. Then we have $L^\infty(K)\subseteq L^p(K)\subseteq L^1(K)$ for all $1\leq p\leq\infty$. Thus, it suffices to prove the assertion for the case $p=1$. Let $f\in{H}_\sigma^1(K)$ and let  $(\phi_\alpha)$ be a bounded approximate identity for $L^1(K)$ such that $\phi_\alpha$ is a bounded continuous function with compact support for all $\alpha$; see \cite[Theorem 1.6.15]{bloom}. Then \cite[Lemma 2.2(i)]{skan} implies that $f*\phi_\alpha$ is continuous and bounded for all $\alpha$. Moreover, it is clear that $f*\phi_\alpha$ is $\sigma$-harmonic, and is therefore constant by Theorem \ref{ami}. This shows that $f$ is also constant, as desired.
\end{proof}

\begin{corollary}
	Let   $K$ be a  compact hypergroup and let $\sigma$ be a non-degenerate probability measure on $K$. Then   any weak$^*$ cluster point $\sigma_0$ of the sequence 
	$\left(\frac{1}{n}\sum_{k=1}^n \sigma^k\right)$ in $M(K)$ is the normalized Haar measure on $K$.
\end{corollary}
\begin{proof}
	Let $\sigma_0$ be a weak$^*$ cluster point of $(\frac{1}{n}\sum_{k=1}^n \sigma^k)$ in $M(K)$. Then we have $\sigma_0$ is an idempotent probability measure on $K$ satisfying $\sigma*\sigma_0=\sigma_0$, and therefore $\sigma_0*f\in {H}_{\sigma}^\infty(K)$ for all $f\in C_b(K)$. It follows from Theorem \ref{ami} and the non-degeneracy of $\sigma$ that for each $f\in C_b(K)$ there exists 
	$\lambda_f\in { C}$ such that $\sigma_0*f=\lambda_f 1$. Let $\omega$ be the normalized Haar measure on $K$. Then  for each $f\in C_b(K)$,
	$$
	\langle \omega, f\rangle=\int_K f(x) d\omega(x)=\int_K(\sigma_0*f)(x)d\omega(x)=\lambda_f.
	$$ 
	Moreover, 
	\begin{eqnarray*}
		\langle \check{\sigma_0}, f\rangle=\langle \check{\sigma_0}*\check{\sigma}, f\rangle=\langle \check{\sigma}, \sigma_0*f\rangle=
		\langle \check{\sigma}, \lambda_f1\rangle=\lambda_f.
	\end{eqnarray*}
	This shows that $\check{\sigma_0}=\omega$. Since $\check{\omega}=\omega$, we conclude that $\sigma_0=\omega$.
\end{proof}

Let $\sigma\in M(K)$. We say that a
measure $\mu\in M(K)$ is $\sigma$-harmonic if it satisfies the convolution equation
$\sigma*\mu=\mu$. Define ${H}_\sigma(K)$ to be the set of all ${\sigma}$-harmonic measures.

\begin{theorem}
	Let $\sigma\in M(K)$ with $\|\sigma\|=1$. Then there is a contractive projection $P_\sigma: M(K) \rightarrow M(K)$ with $P_\sigma(M(K))={H}_\sigma(K)$.
	
\end{theorem}
\begin{proof}
	Let $\mathcal U$ be a free ultra-filter on $\mathbb{N}$, and define 
	$
	P_\sigma: M(K) \rightarrow M(K)
	$
	by the weak$^*$ limit
	$$
	P_\sigma(\mu)=\lim_{\mathcal U}\frac{1}{n}\sum_{k=1}^n \sigma^k*\mu.
	$$
	It is easy to see that $P_\sigma(\mu)=\mu$ for all $\mu\in {H}_\sigma(K)$. Moreover, if $\mu\in M(K)$, then it is easily verified that $\sigma*P_\sigma(\mu)=P_\sigma(\mu)$ and hence $P_\sigma(\mu)\in {H}_\sigma(K)$. These show that $P_\sigma^2=P_\sigma$ and $P_\sigma(M(K))={H}_\sigma(K)$.
\end{proof}
Recall that a measure $\mu\in M(K)$ is non-negative if $\langle \mu, f\rangle\geq 0$ for all $f\in C_0(K)^+$.
\begin{proposition}\label{pos measure}
	Let $\sigma$ be a probability measure on $K$. Then  ${H}_\sigma(K)$ is generated by its non-negative elements.
\end{proposition}
\begin{proof}
	Suppose that  $\mu\in {H}_\sigma(K)$. Then $\sigma*\overline{\mu}=\overline{(\sigma*\mu)}=\overline{\mu}$. This shows that ${H}_\sigma(K)$	is generated by its real measure parts. Now let $\mu\in {H}_\sigma(K)$ be a real measure and let $\mu=\mu_+-\mu_-$, where $\mu_+, \mu_-$ are non-negative measures  in $M(K)$. Since $\sigma$ is positive, the measures $P_\sigma(\mu_+), P_\sigma(\mu_-)\in {H}_\sigma(K)$ are non-negative. Moreover, $\mu=P_\sigma(\mu)=P_\sigma(\mu_+)-P_\sigma(\mu_-)$, which completes the proof.
\end{proof}

\begin{theorem}\label{3.18}
	Let $K$ be non-compact and let $\sigma$ be a non-degenerate probability measure on $K$. Then we have ${H}_\sigma(K)=\{0\}$.
\end{theorem}  
\begin{proof}
	Suppose that $\mu\in H_{\sigma}(K)$ is non-zero. By Proposition \ref{pos measure}, we can assume that $\mu$ is positive. Consider the probability measure $\sigma_0$ defined by
	$\sigma_0:=\sum_{n=1}^\infty\frac{1}{2^n}\sigma^n$. It is clear that $\sigma_{0}\ast\mu=\mu$. By \cite[Theorem 1.6.9]{bloom}, we have that $\delta_{x}\ast\mu=\mu$ for all $x\in {\rm supp}\sigma_{0}=K $. This shows that $\mu$ is the left Haar measure that is finite on $K$, which is a contradiction with non-compactness of $K$. 
	
\end{proof}
\begin{corollary}
	Let $K$ be non-compact and let $\sigma$ be a non-degenerate probability measure on $K$. Then we have ${H}_\sigma^1(K)=\{0\}$.
\end{corollary}
\begin{proof}
	Suppose that $f \in H_{\sigma}^{1}(K)$. Define $\mu:=f\omega$, where $\omega$ is the left Haar measure  on $K$. Then we have
	$$\mu=(\sigma\ast f)\omega=\sigma\ast(f\omega)=\sigma*\mu.$$ Therefore,  
	$\mu=0$ by Theorem \ref{3.18} and hence $f=0$. 
\end{proof}
\begin{theorem}
	let $K$ be compact and let $\sigma$ be an adabted probability measure on $K$. Then we have ${H}_{\sigma}(K)= \mathbb{C}\omega$, where $\omega$ is the normalized Haar measure on $K$.
\end{theorem}
\begin{proof}
	Because $K$ is compact, the Haar measure $\omega$ is in $ M(K)$ and we have $\sigma \ast \omega=\omega$. Therefore, $\mathbb{C}\omega_{K}\subseteq {H}_{\sigma}(K)$. To prove the converse, suppose that $\mu \in {H}_{\sigma}(K)$. Then $\sigma\ast\mu =\mu$. 
	Let $(\phi_{\alpha})$ be a bounded approximate identity for $L^1(K)$ such that $\phi_{\alpha} \in C_{c}^{+}(K)$ for every $\alpha$ and 
	$\phi_{\alpha}\stackrel{w^*}\longrightarrow\delta_{e}$; see   \cite[Theorem 1.6.15]{bloom}.
	Then  $\mu \ast \phi_{\alpha}\stackrel{w^*}\longrightarrow \mu$ and $\mu\ast \phi_{\alpha} \in {H}_{\sigma}^{1}(K)$ for all $\alpha$. Thus,  Theorem \ref{comp} implies that  $\mu \ast \phi_{\alpha}$
	is constant for all $\alpha$. Hence, for every $\alpha$ there is $\lambda_{\alpha} \in \mathbb{C} $ such that $\mu \ast \phi_{\alpha}=\lambda_{\alpha}{1}$.
	It follows that  for each $f \in C_b(K)$, we have
	$$\langle\mu \ast \phi_{\alpha}, f\rangle=\int_{K}\lambda_{\alpha}f d\omega=\langle \lambda_{\alpha}\omega, f\rangle.$$
	Therefore, $\langle\lambda_{\alpha}\omega_{K}, f\rangle\longrightarrow\langle \mu, f\rangle$ for all $f\in C_b(K)$. This shows that there exists $\lambda \in \mathbb{C}$ such that $ \mu=\lambda \omega$. It follows that  ${H}_{\sigma}(K)= \mathbb{C} \omega$.
\end{proof}


\bibliographystyle{amsplain}

\end{document}